\documentclass[preprint,11pt]{elsarticle}

\usepackage{lineno,hyperref}
\journal{Discrete Applied Mathematics}

\usepackage{lipsum}
\usepackage{amssymb,amsmath}

\usepackage{graphicx}
\usepackage[table]{xcolor}%
\usepackage{diagbox}
\usepackage{multicol}
\usepackage{tikz}
\usetikzlibrary{shadows}
\usepackage{pifont}
\usepackage{siunitx}
\usepackage{booktabs}
\usepackage{tabularx}
\usepackage[ruled,vlined,linesnumbered,noend]{algorithm2e}
\usepackage{bm}
\usepackage{multirow}
\usepackage{complexity}
\usepackage{framed}
\usepackage[caption=false]{subfig}
\usepackage{float}
\usepackage{hyperref}
\usepackage{wasysym}
\usepackage[colorinlistoftodos,prependcaption,textsize=tiny]{todonotes}
\usepackage{epstopdf}
\usepackage{algpseudocode}
\usepackage{mathtools}

\ifpdf
  \DeclareGraphicsExtensions{.eps,.pdf,.png,.jpg}
\else
  \DeclareGraphicsExtensions{.eps}
\fi

\definecolor{azul1}{RGB}{183,200,196}

\newtheorem{observation}{Observation}[section]
\newtheorem{theorem}{Theorem}[section]

\newtheorem{lemma}{Lemma}[section]
\newtheorem{corollary}{Corollary}[section]
\newtheorem{conjecture}{Conjecture}[section]

\newenvironment{proof}{\paragraph{Proof}}{\hfill$\square$}

\DeclarePairedDelimiter{\floor}{\lfloor}{\rfloor}

\begin{document}

\begin{frontmatter}

\tnotetext[mytitlenote]{This study was financed by CAPES - Finance Code 001, FAPERJ and CNPq.}

\title{New Results on Edge-coloring and Total-coloring of Split Graphs}

\author[af1]{Fernanda Couto}
\ead{fernandavdc@ufrrj.br}

\author[af2,cor1]{Diego Amaro Ferraz}
\ead{ferrazda@cos.ufrj.br}

\author[af2]{Sulamita Klein}
\ead{sula@cos.ufrj.br}

\address[af1]{Universidade Federal Rural do Rio de Janeiro, 
Nova Igua\c{c}u, Brazil}

\address[af2]{Universidade Federal do Rio de Janeiro, Rio de Janeiro, Brazil}

\cortext[cor1]{Corresponding author}
\begin{abstract} 

A split graph is a graph whose vertex set can be partitioned into a clique and an independent set. A connected graph $G$ is said to be $t$-admissible if admits a special spanning tree in which the distance between any two adjacent vertices is at most $t$. Given a graph $G$, determining  the smallest $t$ for which $G$ is $t$-admissible, i.e., the stretch index of $G$ denoted by $\sigma(G)$, is the goal of the {\sc $t$-admissibility problem}. Split graphs are $3$-admissible and can be partitioned into three subclasses: split graphs with $\sigma=1, 2 $ or $3$. In this work we consider such a partition while dealing with the problem of coloring a split graph. %A $k$-proper coloring of a graph is an assignment of $k$ colors to its elements (vertices or edges) in such a way that adjacent elements receive distinct colors.% 
Vizing proved that any graph can have its edges colored with $\Delta$ or $\Delta+1$ colors, and thus can be classified as \emph{Class 1} or \emph{Class 2}, respectively. When both, edges and vertices, are simultaneously colored, it is conjectured that any graph can be colored with $\Delta+1$ or $\Delta+2$ colors, and thus can be classified as \emph{Type 1} or \emph{Type 2}. Both variants are still open for split graphs. In this paper, using the partition of split graphs presented above, we consider the {\sc edge coloring problem} and the {\sc total coloring problem} for split graphs with $\sigma=2$. For this class, we characterize Class~2 and Type~2 graphs and we provide polynomial-time algorithms to color any Class~1 or Type~1 graph.

\end{abstract}
\end{frontmatter}

\section{Introduction}\label{sec:intro}

The graphs considered in this paper are simple (i.e., without loops and multiple edges), connected, non-directed and unweighted. Basic definitions can be found in~\cite{Bondy}. More specific definitions are given throughout the paper.

% A coloring in a graph is an assignment of colors or labels to its elements. This assignment can be to the vertices of $G$, to the edges of $G$, to both elements simultaneously etc. 
A proper $k$-coloring of a graph is an assignment of $k$ colors to its elements (vertices or edges) in which adjacent or incident elements receive distinct colors. In this case, $G$ is said to be $k$-colorable.
In coloring problems, the goal is usually to minimize the number of assigned colors, that is, to determine the smallest $k$ for which the graph is $k$-colorable. In the \textsc{vertex coloring problem}~\cite{brooks_1941}, this number is called the \emph{chromatic number} and is denoted by $\chi(G)$; in the \textsc{edge coloring problem}~\cite{Vizing}, this number is called the \emph{chromatic index} and is denoted by $\chi'(G)$. It is easy to see that the maximum degree $\Delta$ of a graph $G$ is a lower bound for its chromatic index. Furthermore, Vizing~\cite{Vizing} proved that $\Delta+1$ is an upper bound for the chromatic index of any graph $G$, and thus, graphs can be classified either as \emph{Class 1}, if they are $\Delta$-colorable, or  as \emph{Class 2}, otherwise. After Vizing's result, this problem was also known as the {\sc classification problem}.

When the proposal is to properly color vertices and edges of a graph, simultaneously, we are dealing with the {\sc total coloring problem}~\cite{TotalVizing}. In this case, the smallest number of colors needed in a total coloring of a graph $G$ is called the \emph{total chromatic number} and is denoted by $\chi''(G)$. Again it is easy to see that the value of $\Delta+1$ is a lower bound for the total chromatic number of any graph $G$. The \textit{Total Coloring Conjecture}~(\cite{TCCBehzad}, \cite{TotalVizing}) states that, for any graph $G$, $\chi''(G)=\Delta+1$ or $\chi''(G)=\Delta+2$. Graphs that fit the first case are called Type 1 graphs and graphs that fit the second case are called Type 2 graphs. It is known that some classes satisfy the TCC such as split graphs~\cite{Chen} and indifference graphs~\cite{IndiferencaTCC}. A graph $G$ is called a split graph if $V(G)$ can be partitioned into a clique $X$ and a independent set $Y$. Throughout this text we denote a split graph as $G=((X,Y),E)$ and we consider that $X$ is a maximal clique. Although, \textsc{edge coloring} and \textsc{total coloring problems} are known to be \NP-hard in general (\cite{Holyer1981TheNO}, \cite{SANCHEZARROYO1989315}), both problems remain open when restricted to split graphs. There are several results in the literature that influence (directly or indirectly) the edge (resp. total) coloring of split graphs (\cite{Wilson}, \cite{Plantholt}, \cite{Chen}, \cite{Sheila})  (resp. \cite{Chen}, \cite{Hilton}). Furthermore, until this work, split graphs have been studied in the context of edge and total coloring by considering some subclasses such as split-indifference graphs (for which both versions are solved~\cite{CAMPOS20122690}, \cite{CARMENORTIZ1998209}), split-comparability and split interval graphs for which the edge coloring problem is solved (\cite{ComparabilityOrtiz}, \cite{Gonzaga}). Since the goal is to fully classify split graphs with respect to both variants of the coloring problem, an interesting question that arises is: Which subclasses are left to be studied in order to accomplish this task for split graphs?

The {\sc t-admissibility problem}, which is another quite challenging problem in general, is known to be polynomially solvable for split graphs. Proposed by Cai and Corneil in 1995~\cite{Cail}, the \textsc{$t$-admissibility problem}, also known as the \textsc{minimum stretch spanning tree problem} (\textsc{MSST}), aims to determine the smallest $t$ such that a given graph $G$ admits a tree $t$-spanner, that is, a spanning tree $T$ in which the greatest distance between any pair of adjacent vertices of $G$ is at most $t$. If $G$ admits a tree $t$-spanner, then $G$ is said to be \emph{$t$-admissible} and $t$ is the \emph{stretch factor} associated to the tree. The smallest stretch factor among all spanning trees of $G$ is the \emph{stretch index} of $G$, denoted by $\sigma(G)$, or simply $\sigma$. We call a graph $G$ with $\sigma(G)=t$ a $(\sigma=t)$-graph. Some previous studies motivated the formulation of the \textsc{$t$-admissibility problem}, such as Peleg and Ullman's work~\cite{Peleg} in $1987$, and Chew in $1986$~\cite{Chew}. 
% To verify whether a graph $G$ is $1$-admissible can be done in polynomial time, we can easily see that the graph must be a tree.
Deciding whether a graph $G$ is a $(\sigma=1)$-graph or a $(\sigma=2)$-graph are polynomial-time solvable problems. Indeed, it is easy to see that a graph is a $(\sigma=1)$-graph if, and only if, it is a tree. Moreover, Cai and Corneil provided a linear-time algorithm to decide if a graph has $\sigma(G)=2$. They also settled that, for $t$ at least $4$, the problem is \NP-complete. Curiously, the problem is still open when the goal is to decide if a graph has $\sigma(G)=3$. Split graphs are known to be $3$-admissible~\cite{Panda} and therefore the \textsc{$t$-admissibility problem} partitions the class into $3$ subclasses (see Figure~\ref{particao_admissiblidade}): ($\sigma=1$)-split graphs (bi-stars, i.e., trees with $n$ vertices and at least $n-2$ leaves), ($\sigma=2$)-split graphs or ($\sigma=3$)-split graphs. From the literature~\cite{zbMATH02614481}, we know how to color bi-stars, in both variants (edge and total coloring). Thus, in order to fully classify the problems of edge coloring and total coloring for split graphs, we are left with the study of split graphs with $ \sigma(G)=2$ and $\sigma(G)=3$. 

Although there is a general algorithm~\cite{Cail} that determines whether a graph has stretch index equal to $2$, Theorem~\ref{caracterizacao_split2adm} is a specific result concerning split graphs~\cite{Couto} and is fundamental to the results of this paper. Firstly, we must enunciate the following observation.

\begin{observation}
    \label{obs:pendant_vertex}
    Let $G=(V,E)$ be a graph and let $P~=~\{~v~\in~{V}~|~d(v)~=~1\}$. Then, $\sigma(G)=\sigma(G\setminus{P})$.
\end{observation}

Observation~\ref{obs:pendant_vertex} states that a vertex with degree equal to one does not influence the calculation of the stretch index of a graph, since this edge is in any tree $t$-spanner of the graph. From now on, we will call vertices with degree equal to one as \emph{pendant vertices}. And the edge incident to a pendant vertex is a \emph{pendant edge}.

Theorem~\ref{caracterizacao_split2adm} characterizes $(\sigma=2)$-split graphs.\\

\begin{theorem}\cite{Couto}
    Let $G=((X,Y),E)$ be a split graph such that $\forall{y}\in{Y},~ d_G(y)>1$. Then $\sigma(G)=2$ if, and only, $G$ has a universal vertex, i.e., a vertex which is adjacent to every other vertex of $G$.
    \label{caracterizacao_split2adm}
\end{theorem}

% In other words, any connected split graph is pre-processed by the removal of its pendant vertices. Theorem~\ref{caracterizacao_split2adm} states that, in order to have $\sigma=2$, its is necessary and sufficient that such a graph has a universal vertex. Obviously, this vertex, if it exists, belongs to the clique.

In the next chapters, we characterize $(\sigma=2)$-split graphs which are Class 2 or Type 2, and provide polynomial-time algorithms to color Class 1 and Type 1 graphs.

\begin{figure}[H]
        \label{particao_admissiblidade}
        \includegraphics[scale=0.20]{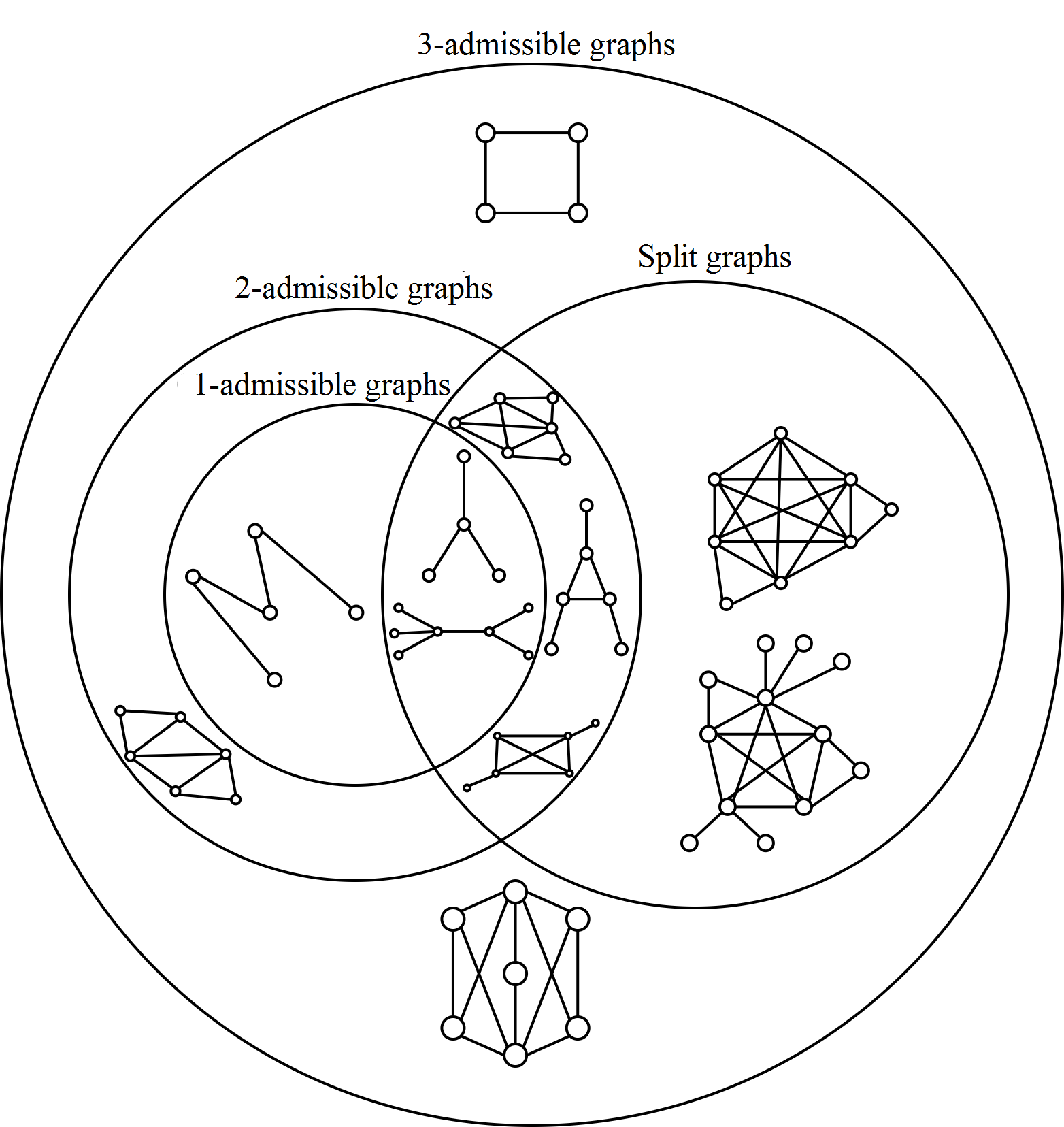}
        \centering
        \caption{Intersection between split graphs and $3$-admissible graphs. Note that any $t~$-~admissible graph is a $k$-admissible graph for $k>t$. But the converse is not true.}
    \end{figure}

\section{Edge-coloring of~\boldmath\texorpdfstring{($\sigma=2$)}{sigma=2}-split graphs}\label{sec:approach}

%\vspace{-0.2cm}
In this section we characterize $(\sigma=2)$-split graphs which are Class~$2$ and we present a polynomial time algorithm to color the edges of $(\sigma=2)$-split graphs that are Class~$1$. Figure~\ref{diagramaCA} illustrates the state of the art of the \textsc{classification problem} considering split graphs. All examples depicted in intersection areas of such a diagram were constructed from some minimal forbidden induced subgraphs of the considered classes. Note that the intersection between split-interval graphs, split non-comparability graphs and split graphs with $\sigma=3$ is empty. The main idea to prove this consists in trying to construct a graph, supposing by contradiction that one exists, from the only minimal forbidden induced subgraph of split comparability graphs that is not forbidden for split-interval graphs, which  is depicted in Figure~\ref{fig:proibido}. Since we do not want to obtain a $(\sigma=2)$-split graph, we must not have a vertex in the clique adjacent to each vertex of the independent set whose degree is equal to $2$. In all possible cases, we obtain either a forbidden induced subgraph for split-interval graphs or a $(\sigma=2)$-split graph, and thus, we conclude that this intersection is empty.

In this work, we are focused on the intersection of the red circle and the blue circle of Figure~\ref{diagramaCA}, i.e., $(\sigma=2)$-split graphs. Although there are some previous results in the literature for split subclasses which have $\sigma=2$, there are infinite examples of split graphs that could not be classified as Class~$1$ or $2$ before this work.
See~\cite{CARMENORTIZ1998209}, \cite{ComparabilityOrtiz}, \cite{Cruz}, \cite{Gonzaga} and \cite{doi:10.1137/1.9780898719796} for references of other graph classes presented in the diagram.

\begin{figure}[H]
  \centering
  \subfloat[]{\includegraphics[width=.23\textwidth]{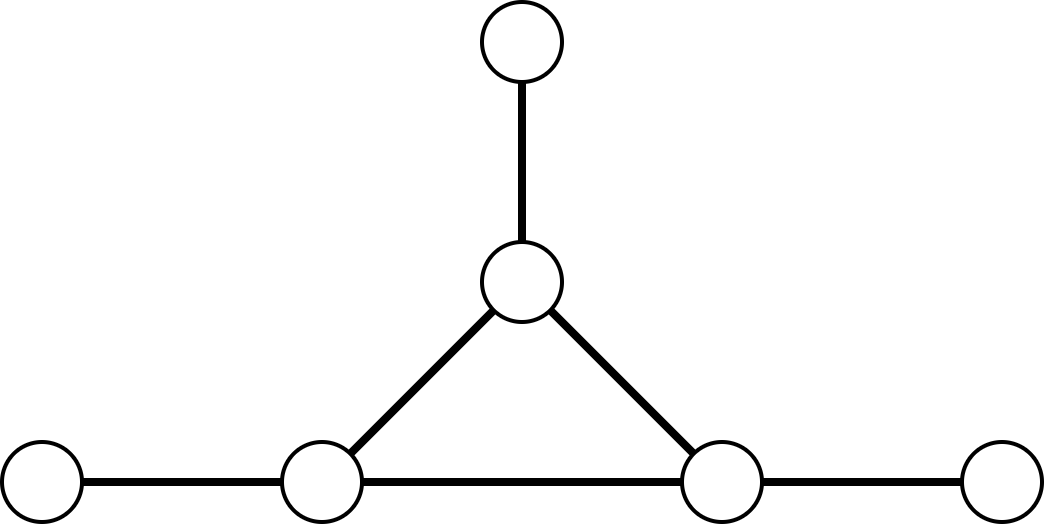}}\hspace{1.5em}%
  \subfloat[]{\includegraphics[width=.23\textwidth]{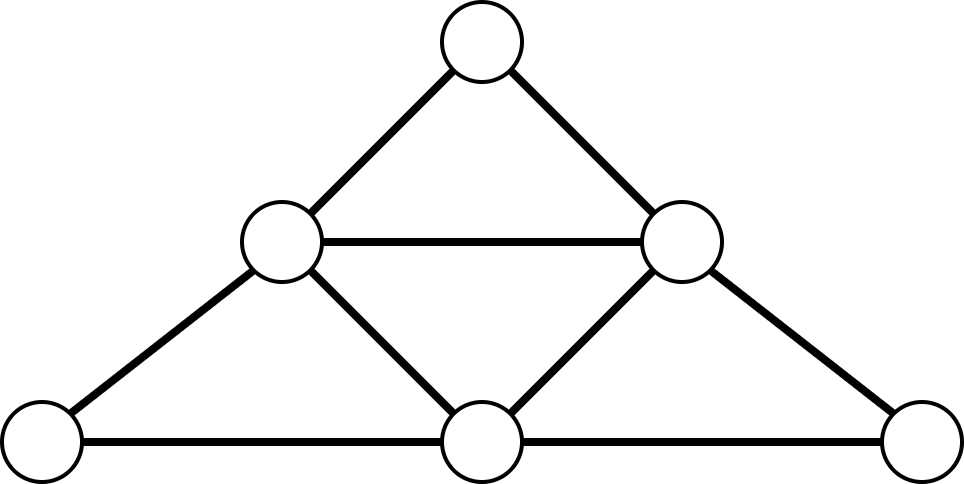}}\hspace{1.5em}\\%
  \subfloat[]{\includegraphics[width=.23\textwidth]{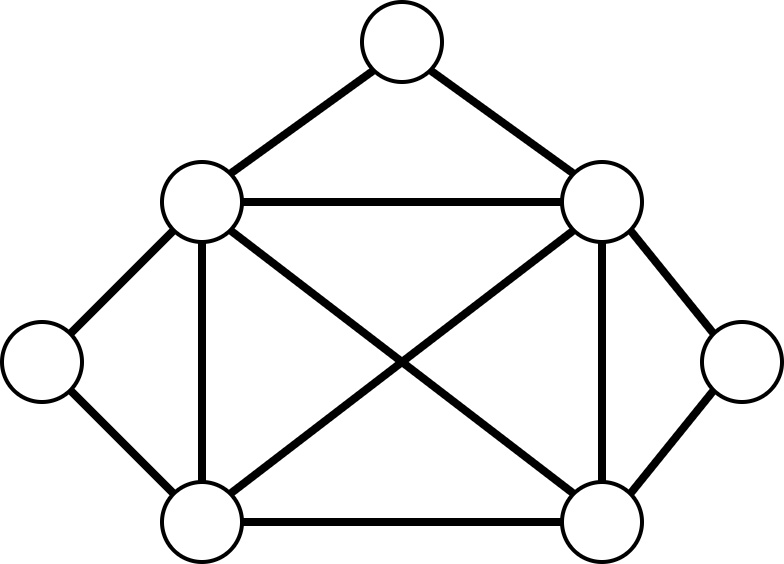}}\hspace{1.5em}%
  \subfloat[]{\includegraphics[width=.23\textwidth]{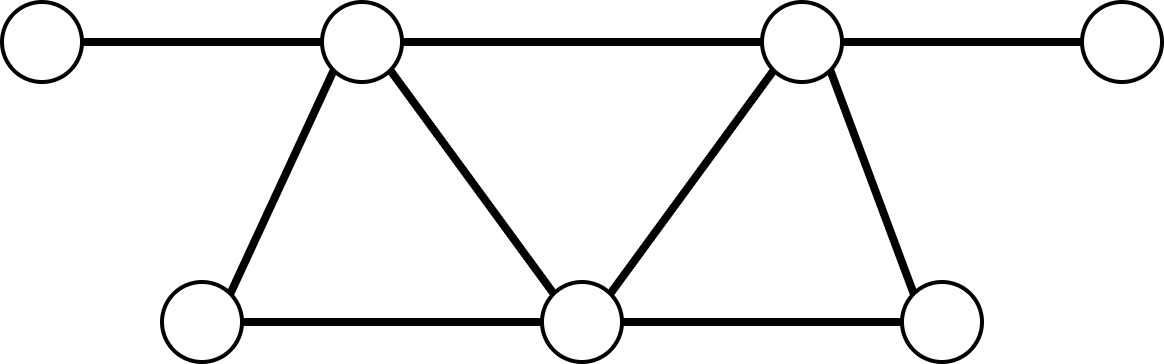}}
  \caption{(a) and (b) Minimal forbidden induced subgraphs which are common to both split-interval and split-comparability graphs. (c) Minimal forbidden induced subgraph for split-interval graphs, exclusively. (d) Minimal forbidden induced subgraph for split-comparability graphs, exclusively.}
  \label{fig:proibido}
\end{figure}    

The following definitions and theorems are needed from now on. 

A graph $G$ is said to be \emph{overfull} if $|E(G)|>\Delta(G)\cdot\floor*{\frac{n}{2}}$, where $n=|V(G)|$. 

A graph $G$ is \emph{subgraph-overfull} if $\exists{H}\subseteq{G}$ such that $\Delta(H)=\Delta(G)$ and $H$ is overfull. A graph $G$ is \emph{neighborhood-overfull} if $H$ is induced by a vertex $v$ such that $d(v)=\Delta(H)$ and by all its neighbors. The Observation~\ref{obs:class2} states a sufficiente condition for a graph $G$ to be Class~2.

\begin{observation}
    If a graph $G$ is overfull, subgraph-overfull or neighborhood-overfull, then $G$ is Class~2.
    \label{obs:class2}
\end{observation}

In particular for split graphs, we have an equivalence of the concepts of being subgraph-overfull and neighborhood-overfull, as stated in Theorem~\ref{thm:eq_split}.

\begin{theorem}\emph{\cite{Figueiredo1995LocalCF}}
    Let $G$ be a split graph. Then, $G$ is subgraph-overfull if and only if $G$ is neighborhood-overfull.
    \label{thm:eq_split}
\end{theorem}

Moreover, the following results concerning the {\sc classification problem} are going to be helpful.

\begin{theorem}\cite{Plantholt}
    Let $G=(V,E)$ be a graph. If $|V|$ is odd and $G$ has a universal vertex, then $G$ is Class~$2$ if and only if $G$ is subgraph-overfull.
    \label{thm:Plantholt}
\end{theorem}

\begin{theorem}\cite{TotalBehzad}
     Let $G=(V,E)$ be a graph. If $|V|$ is even and $G$ has a universal vertex, then $G$ is Class~$1$.
     \label{thm:Behzad}
\end{theorem}

\begin{theorem}\cite{Chen}
    Let $G$ be a split graph. If $G$ has and odd maximum degree, then $G$ is Class~$1$.
\end{theorem}

\begin{figure}[H]
    %\label{diagramaCA}
        \includegraphics[scale=0.55]{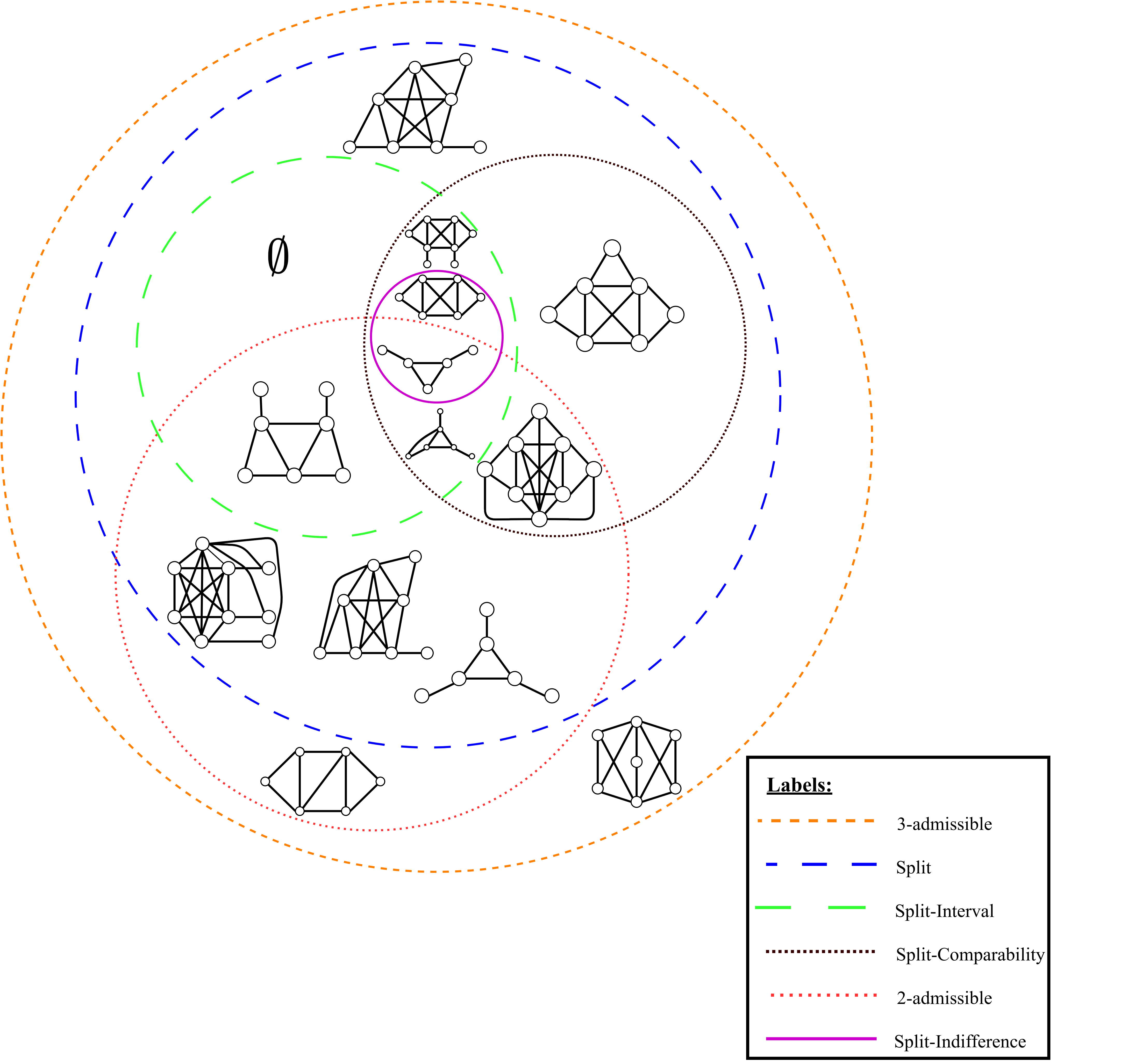}
        \caption{State of the art of the study of the \textsc{classification problem} in split graphs. \label{diagramaCA}}
        \centering
    \end{figure}

Since split graphs with universal vertices and with odd maximum degree are already fully classified, in this work we consider those without a universal vertex and with even maximum degree. As we deal with $(\sigma=2)$-split graphs, by Theorem~\ref{caracterizacao_split2adm}, the graphs we consider have at least one pendant vertex. Lemma~\ref{thrm:pendant} relates the non-existence of pendant neighbors of a vertex $v$ to the property of $N[v]$ inducing an overfull graph, where $N[v]=N(v) \cup \{v\}$ is the closed neighborhood of $v$.

\begin{lemma}
    Let $G=((X,Y),E)$ be a $(\sigma=2)$-split graph with even maximum degree and $v$ be a vertex such that $d(v)=\Delta(G)$. If $G[N[v]]$ is overfull, then $v$ does not have any pendant neighbor.
    \label{thrm:pendant}
\end{lemma}

\begin{proof}
    Suppose, by contradiction, that for every $v$ such that $d(v)=\Delta(G)$, $v$ has at least one pendant vertex in $N_Y(v)$, where $N_Y(v)$ denotes the neighborhood of the vertex $v$ in the independent set $Y$. Consider $|N_Y(v)|=q+1$. Note that $\Delta(G)=(k-1)+q+1=k+q$ (where $|X|=k$) which, by hypothesis, is even. Then $|V(G[N[v]])|=k+q+1$. Since $G[N[v]]$ is overfull, we have that: 
    
    $$|E(G[N[v]])|>\floor*{\frac{k+q+1}{2}}(k+q).$$ 
    
    Since $k+q$ is even, $$\floor*{\frac{k+q+1}{2}}=\frac{k+q}{2}~~~\text{and}~~~E(G[N[v]])|\geq{{\frac{k+q}{2}}(k+q)+1}$$ 
    
    and this implies: 
    
    $$|E(G[N[v]])|\geq{{\frac{(k+q)^2}{2}}+1}.$$
    
    Thus, 
    
    $$|E(G[N[v]])|\geq{{\frac{k^2+2kq+q^2}{2}}+1}.$$ 
    
    Note that this unit added at the end of the expression already represents the edge incident to the considered pendant vertex and therefore, is already omitted in the following expression. The clique $X$ has $\frac{k(k-1)}{2}$ edges, thus we have at least 
    
    $$\frac{k^2+2kq+q^2}{2}-\frac{k^2-k}{2}=\frac{2kq+k+q^2}{2}$$ 
    
    cross edges, i.e., edges between the clique and the neighbors of $v$ in $Y$, where none of these vertices is the considered pendant vertex. Note that, since $X$ is maximal, the greatest number of cross edges incident to a vertex of $Y$ is $(k-1)$. Therefore, as we have $x$ vertices in $G_Y[N[v]]$ that are not necessarily pendant, the maximum number of cross edges is $q(k-1)=kq-q$. Thus,
    
    $$\frac{2kq+k+q^2}{2}-(kq-q)\leq0\Leftrightarrow{q^2+2q+k\leq0}.$$
    
    Solving such inequation, we obtain as roots the values: $-1+\sqrt{1-k}$ and $-(1+\sqrt{1-k})$. The unique case where these expressions assume real values is when $k=1$. As $k$ denotes the size of the maximal clique, we know that $k\geq1$, so the only value that is actually possible for $k$ is $1$. However, when $k=1$, the split graph is a star and consequently, $1$-admissible. This leads us to a contradiction and we conclude that, if $G[N[v]]$ is overfull for some $v\in{X}$ such that $d(v)=\Delta(G)$, then there can not exist any pendant vertex in $N_Y(v)$.
\end{proof}

\vspace{0.5em}
The lemma above leads us to Theorem~\ref{thm:overfull}.

\begin{theorem}
    Let $G=((X,Y),E)$ be a $(\sigma=2)$-split graph, even maximum degree, and let $P$ be the set of pendant vertices of $G$. Then $G$ is neighborhood-overfull if and only if $\exists{v}\in{X}$, such that $v$ is universal in $G[V\setminus{P}]$, $\Delta(G[N[v]])=\Delta(G)$ and $G[N[v]]$ is overfull.
    \label{thm:overfull}
\end{theorem}

\begin{proof}
    Suppose initially that, $\exists{v}\in{X}$ universal in $G[V\setminus{P}]$ such that $\Delta(G[N[v]])=\Delta(G)$ and $G[N[v]]$ is overfull. This is precisely the definition of neighborhood-overfull and this proves the sufficiency of the theorem. It remains to show that if $G$ is neighborhood-overfull, then $\exists{v}\in{X}$ universal in $G[V\setminus{P}]$ such that $\Delta(G[N[v]])=\Delta(G)$ and $G[N[v]]$ is overfull. If $G$ is neighborhood-overfull, then $\exists{x}\in{X}$, such that $d(x)=\Delta(G)$ and $G[N[x]]$ is overfull. As $G$ is a $(\sigma=2)$-split graph, we know that $\exists{v}\in{X}$ which is universal in $G[V\setminus{P}]$. Thus, by Lemma~\ref{thrm:pendant}, $v$ does not have any pendant neighbors.
\end{proof}

\vspace{0.5em}
% The next definitions will be useful to understand some steps of the proposed algorithm.

%\begin{definition}
A color $c$ is \emph{incident} to a vertex $v$ if $v$ or one of the incident edges to $v$ is colored with $c$. We denote by $\mathcal{C}(G)$ the set of colors used to color the elements (vertices or edges) of the graph $G$, $C[v]$ the set of colors incident to the vertex $v$ (it includes the color attributed to the vertex $v$), and $C(v)$ the set of colors used on the edges of $v$. We say that $j$ is a \emph{new color} if $j \notin \mathcal{C}(G)$. On the other hand, we say that a color $c$ is a \emph{missing color} of the vertex $v$ if $c \in \mathcal{C}(G) \setminus C[v]$. 
%\label{def:new_color}
%\end{definition}

\begin{observation}
    Let $G$ be a graph and $u$ a pendant vertex of $G$. The existence of a missing color of $v$ is sufficient to color the pendant edge $uv$.
    \label{obs:pendant_edge}
\end{observation}

\begin{theorem}
    Let $G=((X,Y),E)$ be a $(\sigma=2)$-split graph with even maximum degree. Then $G$ is Class~$2$ if and only if $G$ is neighborhood-overfull. 
    \label{thm:edge_class1}
\end{theorem}

\begin{proof}
    Suppose, initially, that $G$ is neighborhood-overfull. Thus, by Observation~\ref{thm:eq_split}, $G$ is subgraph-overfull, and consequently Class~2, by Observation~\ref{obs:class2}. Next, suppose by contrapositive that $G$ is not neighborhood-overfull. Then, by Theorem~\ref{thm:overfull}, $\forall{v}\in{X}$ universal in $H=G[V\setminus{P}]$, $\Delta(H)~\neq~\Delta(G)$ or $H$ is not overfull.
    The first case we deal is when $\Delta(H)~\neq~\Delta(G)$ and $H$ is overfull. Therefore, $H$ is Class~2, and since $\Delta(H) + 1$ colors are used to color the edges of $H$, for each vertex $x \in X$, there is at least one missing color in $L(x)$, i.e., the list of missing colors of vertex $x$. In order to finish the edge coloring of $G$, it remains to add all pendant vertices to $H$ and color each added pendant edge. Let $v \in X$ be a vertex such that $d(v)=\Delta(G)$ and suppose $d_G(v) - d_H(v) = i$. Since $|L(v)|\geq 1$, it is assured that we use at most $i-1$ new colors to finish such a coloring, and thus, $G$ is Class~1.
    Now, suppose $H$ is not overfull. By Theorems~\ref{thm:Plantholt} and~\ref{thm:Behzad}, $H$ is Class~1. If $\Delta(H)<\Delta(G)$ for each added pendant edge is assigned a new color, and $G$ is Class~1. Otherwise, each $\Delta(G)$-vertex does not have a pendant neighbor. Let $x \in X$ be a vertex such that $\Delta(G)-d(x) > 0$. Note that $|L(x)|= \Delta(G)-d(x)$ and since the existence of a missing color is sufficient to color a pendant edge, by Observation~\ref{obs:pendant_edge}, $G$ is Class~1.
 \end{proof}   

\vspace{0.5em}
Next, we propose an algorithm that colors the edges of any Class~$1$ $(\sigma=2)$-split graph $G$. Let $P$ the subset of pendant vertices of $G$. We first color the edges of the graph $H=G[V \setminus P]$. Then, we add back the pendant vertices to the graph $H$ and perform an edge coloring for every edge incident to the vertices of $P$.\\

\begin{algorithm}[H]
    \label{alg:edge_coloring}
    \caption{Classification of $(\sigma=2$)-split graphs}
    \SetAlgoLined
    \KwData{A $(\sigma=2)$-split graph $G=((X,Y),E)$ and $H=G[V\setminus{P}]$, where $p \in P$ iff $d(p)=1$.}
    \KwResult{A classification for $G$ and a $\Delta$-edge coloring, if $G$ is Class~1.}
    \If{$|E(H)|>\Delta(H)\cdot\floor*{\frac{|V(H)|}{2}}$}{
        \If{$\Delta(H)=\Delta(G)$}{
            \Return {$G$ is Class~2}\;
        }
        \Else{
            \Return {$G$ is Class~1}\;
            Obtain $H_c$ by coloring the edges of $H$ using $\Delta(H)+1$ colors\;
            edge-color($G,H_c$)\;
            \Return{A $\Delta$-edge coloring of $G$}\;
        }
    }
    \Else{
        \Return{$G$ is Class~1}\;
        \If{$\Delta(H)$ is odd}{
            Obtain $H_c$ by coloring the edges of $H$ using Behzad's algorithm\;
        }
        \Else{
            Obtain $H_c$ by coloring the edges of $H$ using Plantholt's algorithm\;
        }
        edge-color($G,H_c$)\;
        \Return{A $\Delta$-edge coloring of $G$}\;
    }   
\end{algorithm}

\begin{procedure}[H]
    \caption{edge-color($G,H$) \label{proc:edge_coloring}}
    \SetAlgoLined
    \KwData{A Class~1 $(\sigma=2)$-split graph $G=((X,Y),E)$ and an edge-colored subgraph $H=G[V\setminus{P}]$, where $p \in P$ iff $d(p)=1$.}
    \ForAll{$v \in V(H)$}{
        $L(v)=\left(\mathcal{C}(H)\setminus C(v)\right)=\{l_1,l_2,...,l_k\}$
    }
    \ForAll{$v \in P$}{
        Add $vx$ to $H$, s.t. $N_G(v)=\{x\}$\;
        \If{$L(x) \neq \emptyset$}{
            $C(v)=\{l_i\}$, $i \in \{1,..,k\}$
            $L(x)=L(x)\setminus \{l_i\}$
        }
        \Else{
            $C(v)=\{j\}$, s.t. $j$ is a new color\;
            \ForAll{$u \in V\setminus \{v,x\}$}{
                $L(u)=L(u)\cup \{j\}$\;
            }
        }
    }
\end{procedure}

% \begin{procedure}[H]
%     \caption{edge-color($G,H$)}
% 	\label{proc:edge_coloring}
%         \begin{algorithmic}[1]
%             \Statex {\it Input:} A Class~1 $(\sigma=2)$-split graph $G=((X,Y),E)$ and an edge-colored subgraph $H=G[V\setminus{P}]$, where $p \in P$ iff $d(p)=1$.
%             %
%             %\Statex {\it Output:} A $\Delta$-edge coloring of $G$.
%             %\Statex \textbf{Procedure} edge-color($G,H$)
%             \State \textbf{for} each vertex $v \in V(H)$
%             \State \indent $L(v)=\left(\mathcal{C}(H)\setminus C(v)\right)=\{l_1,l_2,...,l_k\}$
%             \State \textbf{for} each $v \in P$
%             \State \indent Add $vx$ to $H$, s.t. $N_G(v)=\{x\}$
%             \State \indent \textbf{if} $L(x) \neq \emptyset$ \textbf{then}
%             \State \indent \indent $C(v)=\{l_i\}$, $i \in \{1,..,k\}$
%             \State \indent \indent $L(x)=L(x)\setminus \{l_i\}$
%             \State \indent \textbf{else} 
%             \State \indent \indent $C(v)=\{c\}$, s.t. $c$ is a new color
%             \State \indent \indent \textbf{for} each $u \in V\setminus \{v, x\}$
%             \State \indent \indent \indent $L(u)=L(u)\cup \{c\}$
%             %\State \textbf{Return} A $\Delta$-edge coloring of $G$
%         \end{algorithmic}
%     \end{procedure}  

Figure~\ref{fig:passoapasso} depicts an execution of Algorithm~\ref{alg:edge_coloring}, considering a specific split graph.

\begin{figure}[H]
  \centering
  \subfloat[]{\includegraphics[width=.23\textwidth]{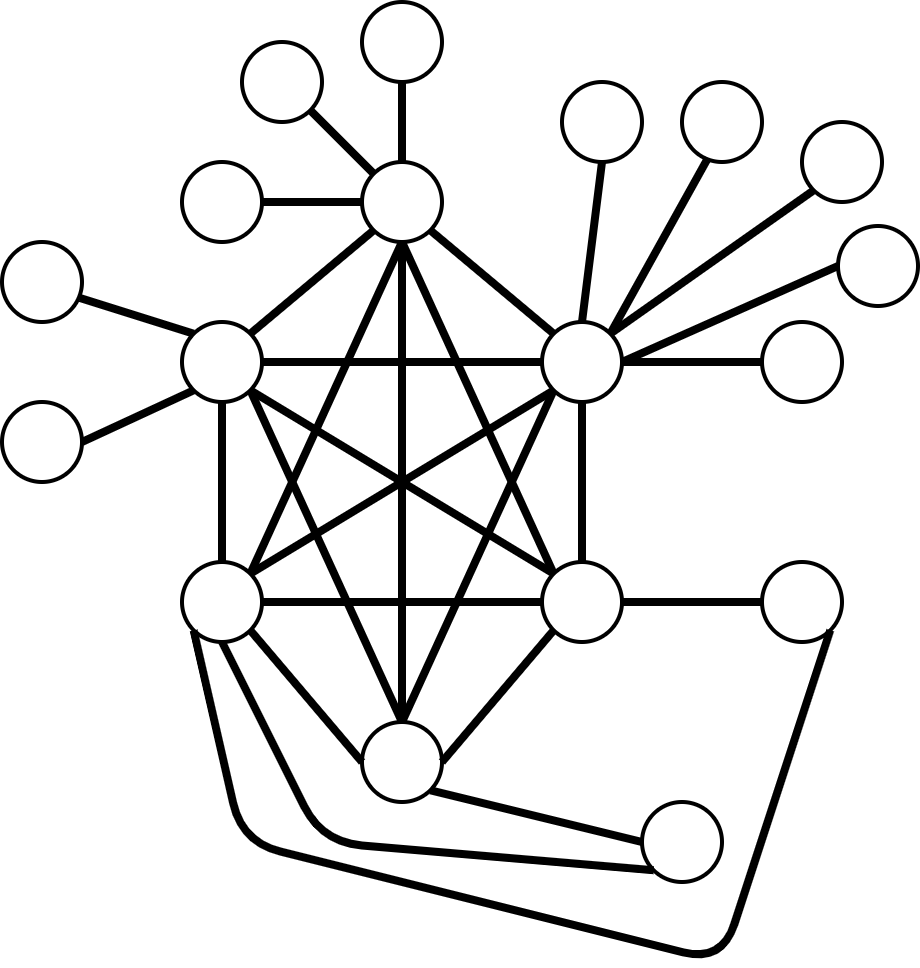}}\hspace{1.5em}%
  \subfloat[]{\includegraphics[width=.18\textwidth]{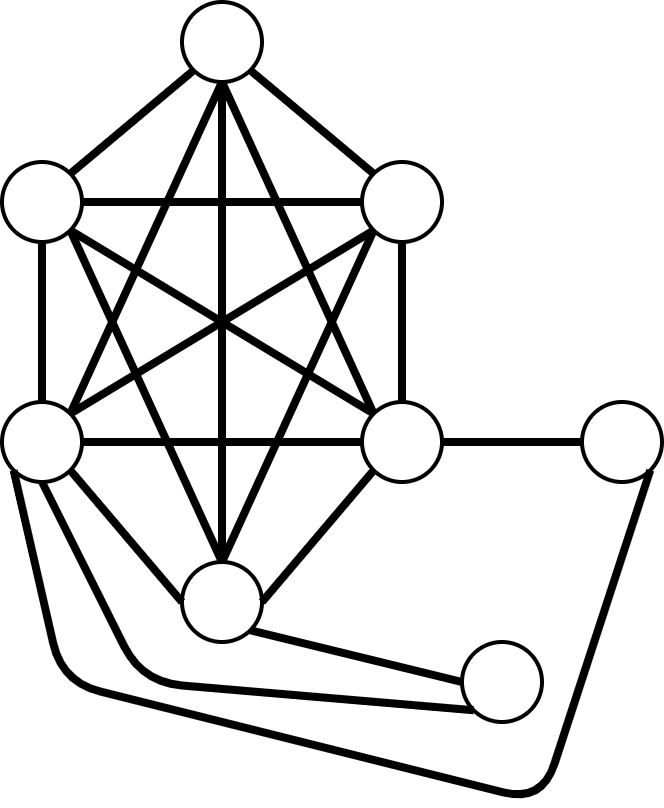}}\hspace{1.5em}%
  \subfloat[]{\includegraphics[width=.18\textwidth]{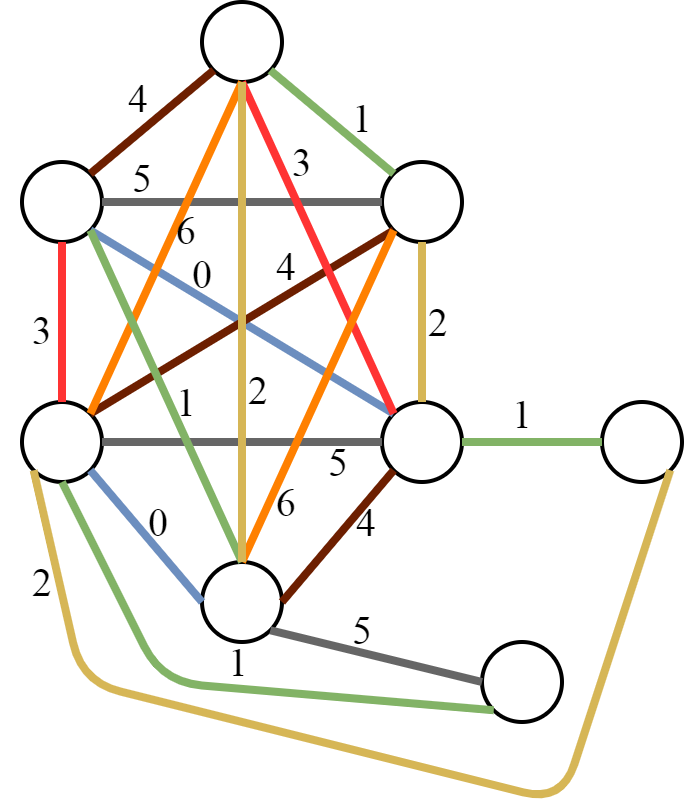}}\\
  \subfloat[]{\includegraphics[width=.25\textwidth]{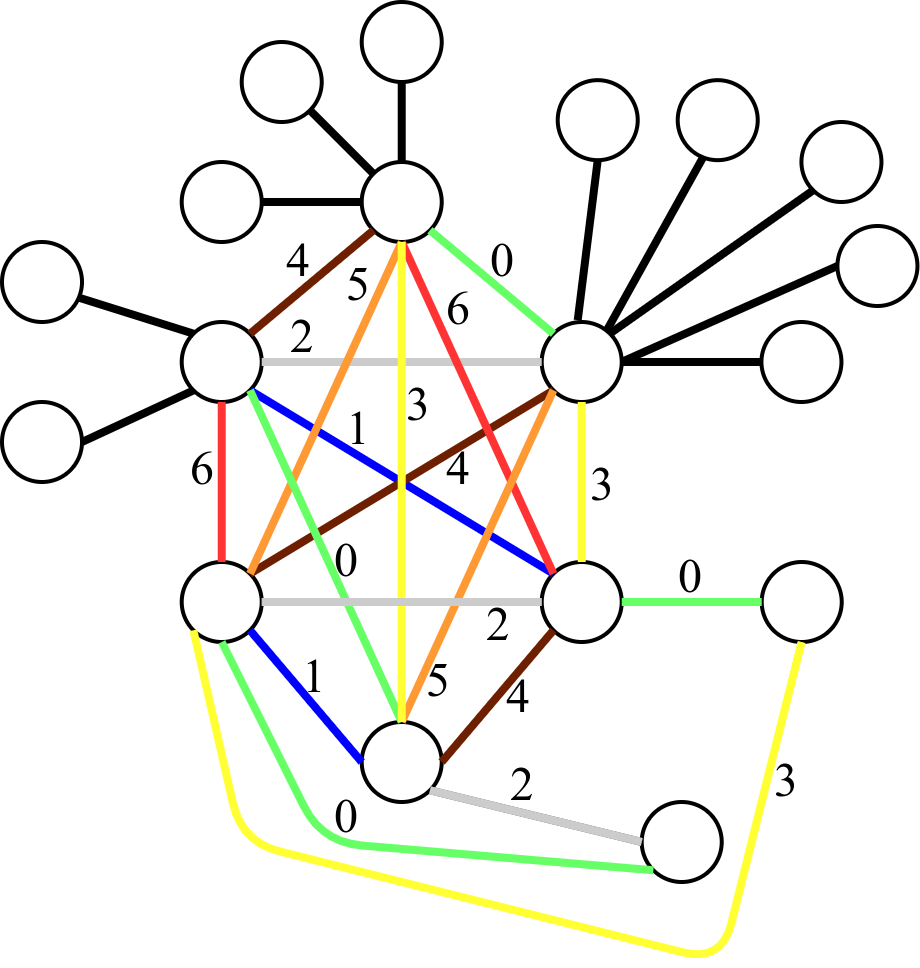}}\hspace{1.5em}%
  \subfloat[]{\includegraphics[width=.25\textwidth]{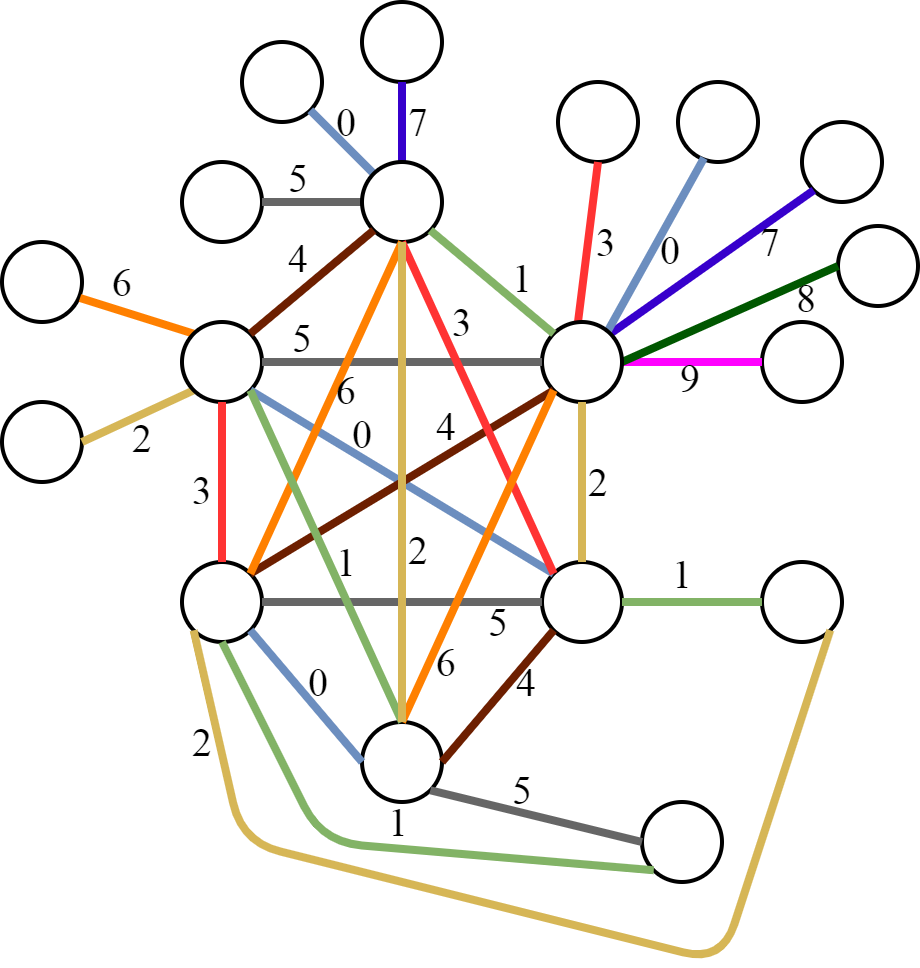}}\hspace{1.5em}%
  \caption{(a) Split Graph $G$ such that $\sigma(G)=2$. (b) Subgraph $H=G[V\setminus P]$. (c) Edge Coloring of $H$. (d) Pendant edges missing coloring. (e) Edge Coloring of $G$.}
  \label{fig:passoapasso}
\end{figure}

\begin{theorem}
Algorithm~\ref{alg:edge_coloring} is correct and runs in polynomial time.\newline
\label{thm:correction_edge} 
\end{theorem}
\vspace{-2em}
\begin{proof}
    The correction of Algorithm~\ref{alg:edge_coloring} follows from Theorems~\ref{thm:Plantholt}, \ref{thm:Behzad}, \ref{thm:overfull} and \ref{thm:edge_class1}. Note that the time complexity of Algorithm~\ref{alg:edge_coloring} depends on the complexity of the algorithm used to color the subgraph $H$. Moreover, Procedure~\ref{proc:edge_coloring} runs in $\mathcal{O}(n^2)$.
\end{proof}
\newline

Theorem~\ref{thm:edge_class1} and Theorem~\ref{thm:eq_split} corroborate two edge-coloring famous conjectures considering $(\sigma=2)$-split graphs.

The first, and best known, is the Hilton Conjecture~\cite{ConjecturaHilton}. 

\begin{conjecture}\emph{\cite{ConjecturaHilton}}
   Let $G$ be a graph with $\Delta(G)>\frac{|V|}{3}$. $G$ is Class $1$ if and only if $G$ is not subgraph-overfull.
    \label{con:overfull}
\end{conjecture}

Note that the characterization provided by Theorem~\ref{thm:edge_class1} holds for any $(\sigma=2)$-split graph, and thus it also holds for those with $\Delta>\frac{|V|}{3}$ as stated by Hilton.

Corollary~\ref{cor:Hilton} follows from Theorem~\ref{thm:edge_class1}.

\begin{corollary}
    The class of $(\sigma=2)$-split graphs satisfies Conjecture~\ref{con:overfull}.
    \label{cor:Hilton}
\end{corollary}

%\textcolor{white}{line}\\

Figueiredo et. al, stated the following conjecture concerning chordal graphs, i.e., graphs without induced cycles of length at least $4$.

\begin{conjecture}\emph{\cite{Figueiredo1995LocalCF}}
    Let $G$ be a chordal graph. Then $G$ is Class~$2$ if, and only if, $G$ is subgraph-overfull.
    \label{con:Celina}
\end{conjecture}

Since split graphs are chordal graphs and, by Theorem~\ref{thm:eq_split}, being subgraph-overfull and neighborhood-overfull are equivalent concepts when restricted to split graphs, $(\sigma=2)$-split graphs also satisfy Conjecture~\ref{con:Celina}.

\begin{corollary}
    The class of $(\sigma=2)$-split graphs satisfies Conjecture~\ref{con:Celina}.
\end{corollary}

% As a direct consequence of Lemma~\ref{thrm:pendant}, we can rewrite Theorem~\ref{thm:edge_class1} as follows:

% \begin{corollary}
%     Let $G=((X,Y),E)$ be $(\sigma=2)$-split graph, with even $\Delta(G)$ and $A=\{a_i~|~d(a_i)=\Delta(G)\}$. Then $G$ is Class~$2$ if and only if $\exists~{a_i}$ without pendant vertices and $G[N[a_i]]$ is overfull.
% \end{corollary}

Note that, for the recognition of $(\sigma=2)$-split graphs with even $\Delta(G)$ that are not neighborhood-overfull, we can, by Lemma~\ref{thrm:pendant}, exclude some vertices of maximum degree, since only universal vertices in $G[V\setminus{P}]$ are candidates to induce, with their neighborhood, an overfull subgraph. Furthermore, if all these candidates have pendant vertices, still by Lemma~\ref{thrm:pendant}, we can conclude that $G$ is not an overfull subgraph. Consequently, by Theorem~\ref{thm:edge_class1}, we can state the following sufficient condition.

\begin{corollary}
    Let $G=((X,Y),E)$ be a split graph $\sigma(G)=2$, $A=\{a_i|d(a_i)=\Delta(G)$ and $a_i$ is universal in $G[V\setminus P]\}$. If every $a_i$ has one or more pendant vertices, $G$ is Class~$1$.
\end{corollary}

%\section{Total coloring of split graphs with $\sigma=2$}
\section{Total-coloring of~\boldmath\texorpdfstring{($\sigma=2$)}{sigma=2}-split graphs}
In this section we deal with the {\sc total coloring problem} using the same approach we used for the {\sc edge coloring problem}. Next, we present two theorems that allow us to classify $(\sigma=2)$-split graphs with respect to the \textsc{total coloring problem}.

\begin{theorem}\emph{\cite{Chen}} 
    Let $G=((X,Y);E)$ be a split graph. If $G$ has even maximum degree $\Delta(G)$, then $G$ is Type~{1}.
    \label{thm:Chen}
\end{theorem}

\begin{theorem}(Hilton's Condition)~\emph{\cite{Hilton}} 
    Let $G$ be a graph with an even number of vertices. If $G$ has a universal vertex, then $G$ is Type~{2} if and only if $|E(\overline{G})|+\alpha'(\overline{G})<\frac{|V(G)|}{2}$, where $\alpha'(\overline{G})$ denotes the size of the maximum independent sets of edges of $\overline{G}$.
    \label{thm:Hilton_condition}
\end{theorem}

Theorem~\ref{thm:total_type1} characterizes Type~2 $(\sigma=2)$-split graphs with odd maximum degree.

\begin{theorem}
    Let $G=((X,Y),E)$ be a $(\sigma=2)$-split graph with odd maximum degree and $H=G(V \setminus P)$. Then $G$ is Type~2 if, and only if, $\Delta(H) =\Delta(G)$ and $H$ satisfies Hilton's condition.
    \label{thm:total_type1}
\end{theorem}

\begin{proof}
    If $H$ satisfies Hilton's condition and $\Delta(H)=\Delta(G)$, then $H$ is Type~2 and thus $G$ is Type~2. Next, suppose by contrapositive that $\Delta(H) \neq \Delta(G)$ or $H$ does not satisfy Hilton's condition. Firstly, suppose $\Delta(H) \neq \Delta(G)$ and $H$ satisfies Hilton's condition. Therefore, $H$ is Type~2, and since $\Delta(H)+2$ colors are used to color the elements of $H$, for each vertex $x \in X$, there is at least one missing color in $L(x)$. In order to finish the total coloring of $G$, it remains to add all pendant vertices to $H$ and color each added pendant vertex and edge. Let $w \in X$ be a vertex such that $d(w)=\Delta(G)$ and suppose $d_G(w) - d_H(w) = i$. Since $|L(w)|\geq 1$, it is assured that we use at most $i-1$ new colors to color the pendant edges. To finish the total-coloring of $G$, we assign to each pendant vertex $v$ any color $c \in C(x)$, where $N_G(v)=\{x\}$. Therefore, $G$ is Type~1. Now, suppose $H$ does not satisfy Hilton's condition. If $\Delta(H)=\Delta(G)$, then $\chi''(H)=\Delta(H)+1$ and for each vertex $u$ such that $d(u)<\Delta(G)$, $|L(u)|\geq 1$. By Observation~\ref{obs:pendant_edge}, the missing colors in $L(u)$ are sufficient to color each added pendant edge of vertex $u$. Pendant vertices are colored the same way as in the previous case. Thus, $G$ is Type~1. On the other hand, if $\Delta(H)\neq\Delta(G)$, new colors are necessary to color all added pendant edges. Let $w \in X$ be a vertex such that $d(w)=\Delta(G)$ and suppose $d_G(w) - d_H(w) = i$. Since $|L(w)|\geq 1$, it is assured that we use at most $i$ new colors to color the pendant edges. Again, the coloring of pendant vertices follows as the previous cases. Therefore, $G$ is Type~1.
\end{proof}

\vspace{0.5em}   
Next we present an algorithm that performs the total coloring of Type~$1$ $(\sigma=2)$-split graphs.\\

% The algorithm works, essentially, the same way as Algorithm~\ref{alg:edge_coloring} until Phase~$3$. Algorithm~\ref{alg:total_coloring} is exactly  Algorithm~\ref{alg:edge_coloring} but phase four, in which pendant vertices are colored. The way phase 1 is performed also depends on total coloring results for graphs with a universal vertex.\\ 

\begin{algorithm}[H]
    \caption{Total coloring of $(\sigma=2$)-split graphs \label{alg:total_coloring}}
    \SetAlgoLined
    \KwData{A $(\sigma=2)$-split graph $G=((X,Y),E)$ and $H=G[V\setminus{P}]$, where $p \in P$ iff $d(p)=1$.}
    \KwResult{The type of $G$ and a $(\Delta+1)$-total coloring, if $G$ is Type~1.}
    \If{$\Delta(H)$ is even }{
        \If{$\Delta(H)=\Delta(G)$}{
            \Return{$G$ is Type~1}\;
        }
        \Else{
            \Return{$G$ is Type~1}\;
            Obtain $H_c$ using Chen's algorithm\;
            total-color($G,H_c$)\;
        \Return{a $(\Delta+1)$-total coloring of $G$}\;
        }
    }
    \Else{
        \If{$H$ satisfies the Hilton's Condition}{
              \If{$\Delta(H)=\Delta(G)$}{
              \Return{$G$ is Type~2}\;
              }
              \Else{
                \Return{$G$ is Type~1}\;
                Obtain $H_c$ using~\cite{Chen}\;
                total-color($G,H_c$)\;
                \Return{a $(\Delta+1)$-total coloring of $G$}\;
              }
        }
        \Else{
            \Return{$G$ is Type~1}\;
            Obtain $H_c$ using Hilton's algorithm\;
            total-color($G,H_c$)\;
            \Return{a $(\Delta+1)$-total coloring of $G$}\;
        }
    }
\end{algorithm}
\begin{procedure}[H]
    \caption{total-color($G,H$) \label{proc:total_coloring}}
    \SetAlgoLined
    \KwData{A Type~1 $(\sigma=2)$-split graph $G=((X,Y),E)$ and an edge-colored subgraph $H=G[V\setminus{P}]$, where $p \in P$ iff $d(p)=1$.}
    \ForAll{$v \in V(H)$}{
        $L(v)=\left(\mathcal{C}(H)\setminus C(v)\right)=\{l_1,l_2,...,l_k\}$
    }
    \ForAll{$v \in P$}{
        Add $vx$ to $H$, s.t. $N_G(v)=\{x\}$\;
        Color $v$ with any $c \in C(x)$\;
        \If{$L(x) \neq \emptyset$}{
            $C(v)=\{l_i\}$, $i \in \{1,..,k\}$
            $L(x)=L(x)\setminus \{l_i\}$
        }
        \Else{
            $C(v)=\{j\}$, s.t. $j$ is a new color\;
            \ForAll{$u \in V\setminus \{v,x\}$}{
                $L(u)=L(u)\cup \{j\}$\;
            }
        }
    }
\end{procedure}

\textcolor{white}{line}\\    
Figure~\ref{fig:passoapasso2} shows an application of Algorithm~\ref{alg:total_coloring} in order to perform the total-coloring of Type~$1$ $(\sigma=2)$-split graphs:

\begin{figure}[H]
  \centering
  \subfloat[]{\includegraphics[width=.20\textwidth]{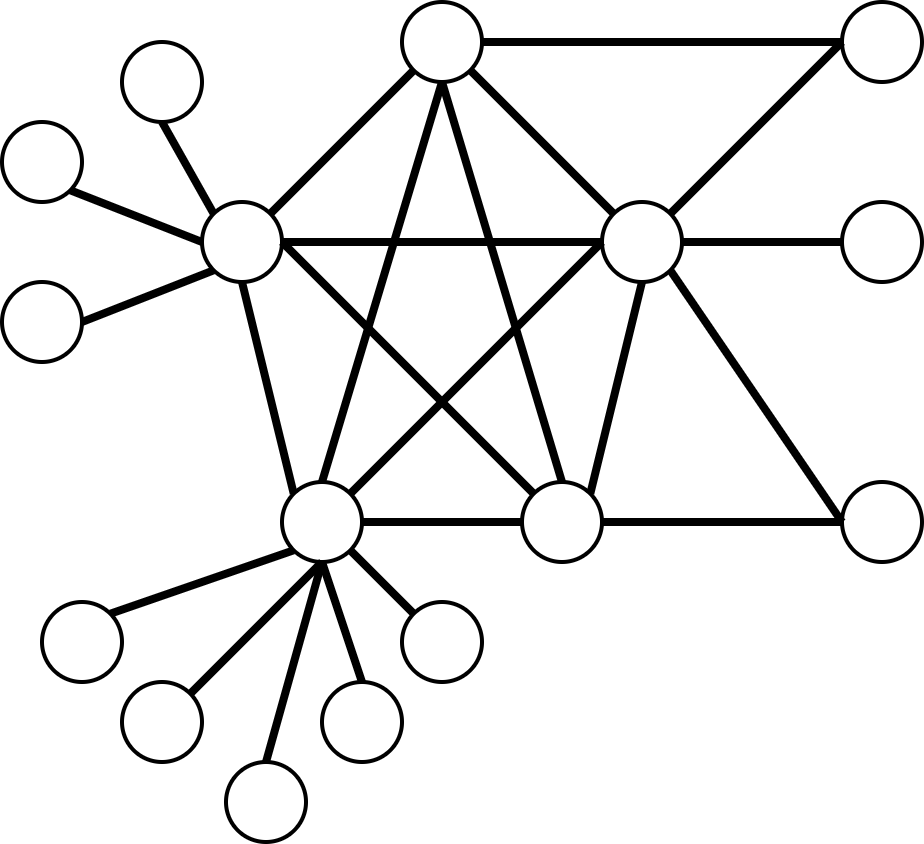}}\hspace{1.5em}%
  \subfloat[]{\includegraphics[width=.18\textwidth]{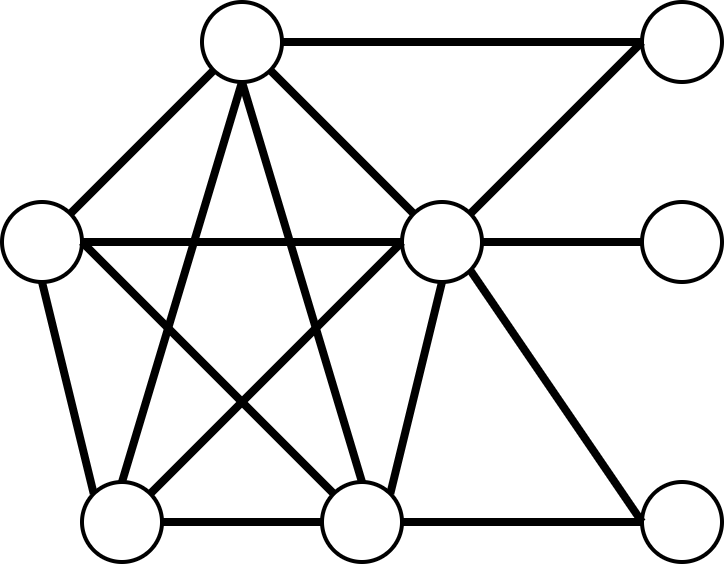}}\hspace{1.5em}%
  \subfloat[]{\includegraphics[width=.18\textwidth]{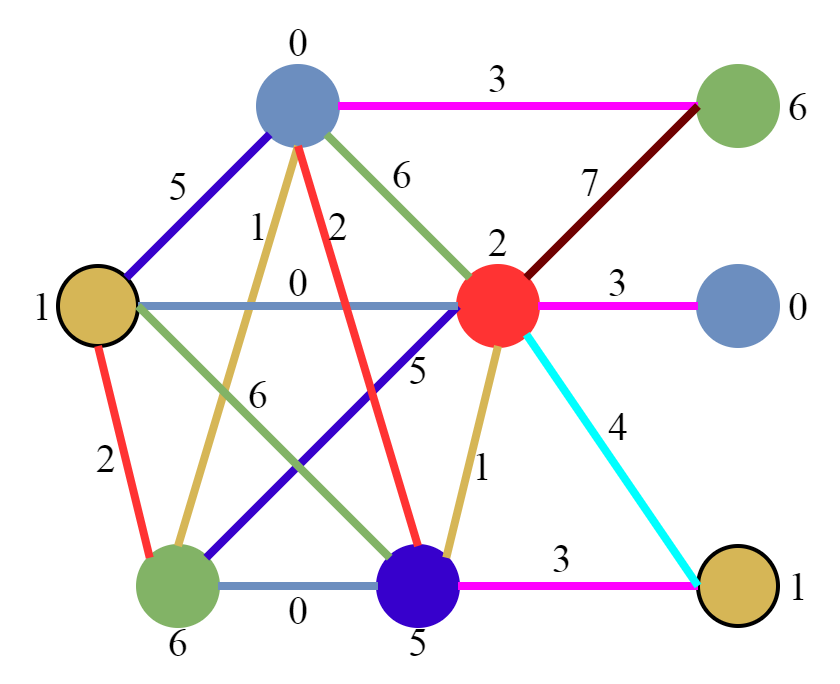}}\\%\hspace{1.5em}%
  \subfloat[]{\includegraphics[width=.20\textwidth]{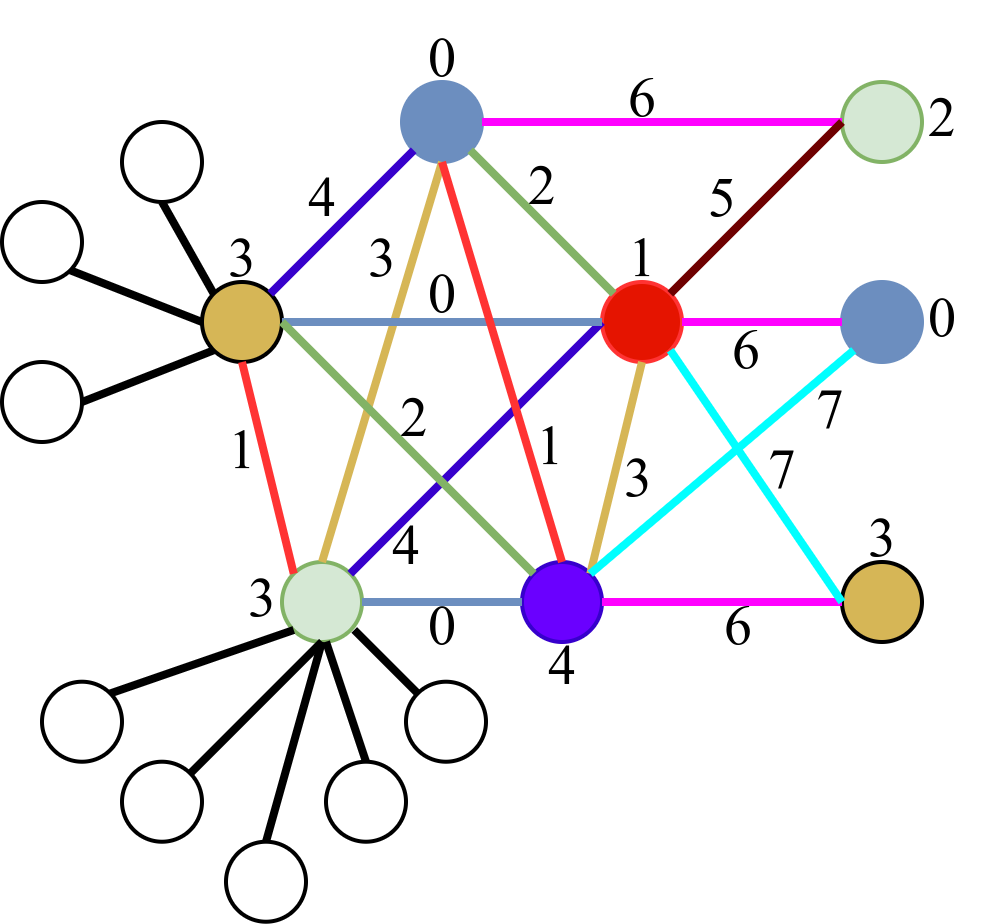}}\hspace{1.5em}%
  \subfloat[]{\includegraphics[width=.20\textwidth]{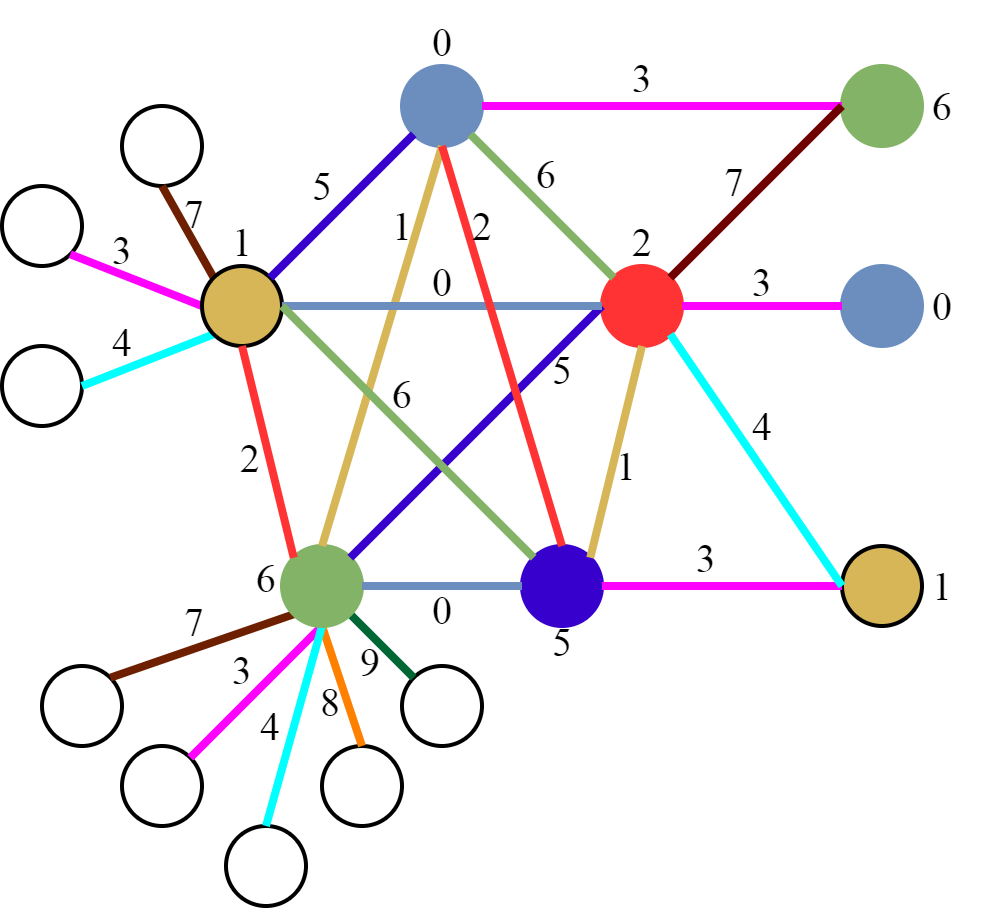}}\hspace{1.5em}%
  \subfloat[]{\includegraphics[width=.20\textwidth]{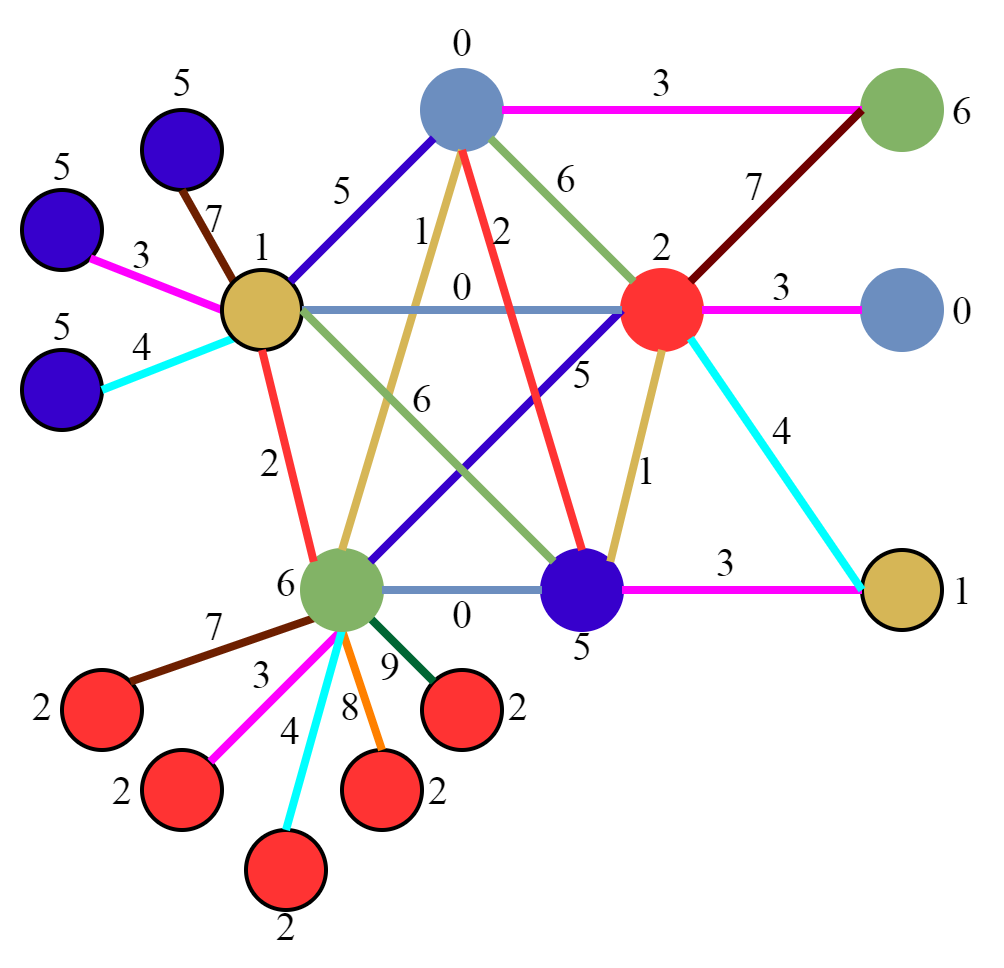}}\hspace{1.5em}%
  \caption{(a) Split graph $G$ such that $\sigma(G)=2$. (b) Subgraph $H=G[V \setminus P]$. (c) Total-coloring of $H$. (d) Pendant edges and pendant vertices missing coloring. (e) Edge-coloring of pendant edges. (f) Total-coloring of $G$.}
  \label{fig:passoapasso2}
\end{figure}

\begin{theorem}
 Algorithm~\ref{alg:total_coloring} is correct and runs in polynomial time.
    \label{thm:correction_total}
\end{theorem}

\begin{proof}
    The correctness of Algorithm~\ref{alg:total_coloring} follows from Theorems~\ref{thm:Chen},~\ref{thm:Hilton_condition} and~\ref{thm:total_type1}. Note that the time complexity of Algorithm~\ref{alg:total_coloring} depends on the complexity of the algorithm used to color the subgraph $H$. Moreover, Procedure~\ref{proc:total_coloring} runs in $\mathcal{O}(n^2)$.
\end{proof}

\section{Concluding remarks and further work}\label{sec:conc}

For several years, the coloring problem and its variations for split graphs have been studied mainly considering some of its subclasses. In this work we considered split graphs in the context of the \textsc{$t$-admissibility problem}, which gave us a new perspective for dealing with the \textsc{edge coloring problem} and the \textsc{total coloring problem}. We characterized Class~2 and Type~2 $(\sigma=2)$-split graphs and provided polynomial-time algorithms to color Class~1 and Type~1 $(\sigma=2)$-split graphs. Next step consists in finishing the study of $(\sigma=3)$-split graphs, in this way, obtaining a fully classify split graphs considering edge and total coloring problems.

\bibliographystyle{abbrv}
\bibliography{AMC23}

\begin{thebibliography}{10}

\bibitem{Sheila}
S.~Almeida.
\newblock {\em Coloração de Arestas em Grafos Split}.
\newblock PhD thesis, Universidade Estadual de Campinas, Instituto de Computação, 03 2012.

\bibitem{TCCBehzad}
M.~Behzad.
\newblock {\em Graphs and their chromatic numbers}.
\newblock PhD thesis, Michigan State University, 1965.

\bibitem{TotalBehzad}
M.~Behzad, G.~Chartrand, and J.~Cooper.
\newblock The colour numbers of complete graphs.
\newblock {\em J. London Math. Soc.}, 42, 01 1967.

\bibitem{Bondy}
J.~A. Bondy and U.~S.~R. Murty.
\newblock {\em Graph Theory with Applications}.
\newblock Elsevier, New York, 1976.

\bibitem{doi:10.1137/1.9780898719796}
A.~Brandstädt, V.~B. Le, and J.~P. Spinrad.
\newblock {\em Graph Classes: A Survey}.
\newblock Society for Industrial and Applied Mathematics, 1999.

\bibitem{brooks_1941}
R.~L. Brooks.
\newblock On colouring the nodes of a network.
\newblock {\em Mathematical Proceedings of the Cambridge Philosophical Society}, 37(2):194–197, 1941.

\bibitem{Cail}
L.~Cai and D.~G. Corneil.
\newblock Tree spanners.
\newblock {\em SIAM J. Discrete Math.}, 8:359--387, 08 1995.

\bibitem{CAMPOS20122690}
C.~Campos, C.~{de Figueiredo}, R.~Machado, and C.~{de Mello}.
\newblock The total chromatic number of split-indifference graphs.
\newblock {\em Discrete Mathematics}, 312(17):2690--2693, 2012.
\newblock Proceedings of the 8th French Combinatorial Conference.

\bibitem{Chen}
B.-L. Chen, H.-L. Fu, and M.-T. Ko.
\newblock Total chromatic number and chromatic index of split graphs.
\newblock {\em JCMCC. The Journal of Combinatorial Mathematics and Combinatorial Computing}, 17, 01 1995.

\bibitem{ConjecturaHilton}
A.~Chetwynd and A.~Hilton.
\newblock Star multigraphs with three vertices of maximum degree.
\newblock {\em Mathematical Proceedings of the Cambridge Philosophical Society}, 100:303 -- 317, 09 1986.

\bibitem{Chew}
L.~P. Chew.
\newblock There are planar graphs almost as good as the complete graph.
\newblock {\em Journal of Computer and System Sciences}, 39:205--219, 01 1986.

\bibitem{Couto}
F.~Couto and L.~Cunha.
\newblock Hardness and efficiency on minimizing maximum distances in spanning trees.
\newblock {\em Theoretical Computer Science}, 838, 06 2020.

\bibitem{Figueiredo1995LocalCF}
C.~M.~H. de~Figueiredo, J.~Meidanis, and C.~P. de~Mello.
\newblock Local conditions for edge-coloring.
\newblock In {\em Journal of Combinatorial Mathematics and Combinatorial Computing}, volume~32, pages 79--92, 2000.

\bibitem{IndiferencaTCC}
C.~M.~H. de~Figueiredo, J.~a. Meidanis, and C.~P. de~Mello.
\newblock Total-chromatic number and chromatic index of dually chordal graphs.
\newblock {\em Inf. Process. Lett.}, 70(3):147–152, may 1999.

\bibitem{Cruz}
J.~B. de~Sousa~Cruz, C.~N. da~Silva, and S.~M. de~Almeida.
\newblock The overfull conjecture on split-comparability graphs.
\newblock {\em arXiv: Combinatorics}, 2017.

\bibitem{Gonzaga}
L.~Gonzaga.
\newblock Coloração de arestas em grafos split-comparabilidade e split-intervalos.
\newblock Master's thesis, Universidade Estadual de Campinas, Instituto de Computação, 2021.

\bibitem{Hilton}
A.~Hilton.
\newblock A total chromatic number analogue of plantholt's theorem.
\newblock {\em Discrete Mathematics}, 79:169–175, 01 1990.

\bibitem{Holyer1981TheNO}
I.~Holyer.
\newblock The np-completeness of edge-coloring.
\newblock {\em SIAM J. Comput.}, 10:718--720, 1981.

\bibitem{zbMATH02614481}
D.~{K\"onig}.
\newblock {Graphok \'es alkalmaz\'asuk a determin\'ansok \'es a halmazok elm\'elet\'ere. (\"Uber Graphen und ihre Anwendung auf Determinantentheorie und Mengenlehre.).}
\newblock {\em {Mat. Term\'eszett. \'Ertes.}}, 34:104--119, 1916.

\bibitem{CARMENORTIZ1998209}
C.~Ortiz, N.~Maculan, and J.~L. Szwarcfiter.
\newblock Characterizing and edge-colouring split-indifference graphs.
\newblock {\em Discrete Applied Mathematics}, 82(1):209--217, 1998.

\bibitem{ComparabilityOrtiz}
C.~Ortiz and M.~Villanueva.
\newblock On split-comparability graphs.
\newblock {\em In Proc. II ALIO-EURO Workshop on Pratical Combinatorial Optmization}, pages 91--105, 1996.

\bibitem{Panda}
B.~Panda and A.~Das.
\newblock Tree 3-spanners in 2-sep chordal graphs: Characterization and algorithms.
\newblock {\em Discrete Applied Mathematics}, 158:1913--1935, 10 2010.

\bibitem{Peleg}
D.~Peleg and J.~Ullman.
\newblock An optimal synchronizer for the hypercube.
\newblock {\em In Proceedings of the 6th ACM Symposium on Principles of Distributed Computing}, 18:77--85, 01 1987.

\bibitem{Plantholt}
M.~Plantholt.
\newblock The chromatic index of graphs with a spanning star.
\newblock {\em J. Graph Theory}, 5:45--53, 01 1981.

\bibitem{SANCHEZARROYO1989315}
A.~Sánchez-Arroyo.
\newblock Determining the total colouring number is np-hard.
\newblock {\em Discrete Mathematics}, 78(3):315--319, 1989.

\bibitem{Vizing}
V.~Vizing.
\newblock On an estimate of the chromatic class of a $p$-graph (in russian).
\newblock {\em (Russian) Diskret Analiz}, 3, 01 1964.

\bibitem{TotalVizing}
V.~Vizing.
\newblock Some unsolved problems in graph theory (russian).
\newblock {\em Russian Mathematical Surveys - RUSS MATH SURVEY-ENGL TR}, 23:125--141, 12 1968.

\bibitem{Wilson}
R.~Wilson and J.~Watkins.
\newblock Graphs: An introductory approach.
\newblock {\em The Mathematical Gazette}, 75, 03 1991.

\end{thebibliography}

\end{document}